\documentclass[10pt]{article}
\usepackage{amsfonts}
\usepackage{amsmath,amssymb}
\usepackage[dvips]{graphicx}
\usepackage{algorithm}
\usepackage{algorithmicx}
\usepackage{algpseudocode}
\usepackage{yhmath}
\usepackage{qtree}
\usepackage{xcolor}
\usepackage{epsfig}
\usepackage{lineno,hyperref}
\usepackage{mathtools}
\usepackage{amsthm}
\usepackage[T1]{fontenc}
\usepackage[utf8]{inputenc}
\usepackage{booktabs}
\usepackage{authblk}
\usepackage[toc]{appendix}
\numberwithin{equation}{section}
\usepackage{verbatim}\usepackage{graphicx}
\usepackage{geometry}
\usepackage[many]{tcolorbox}
\usetikzlibrary{shadows}
\usepackage{caption}
\usepackage{amsmath}
\allowdisplaybreaks[1]

\newtcolorbox{shadedbox}{
  breakable,
  enhanced jigsaw,
  colback=white,
}

\newcommand{\be}{\begin{equation}}
\newcommand{\ee}{\end{equation}}
\newcommand{\ba}{\begin{array}}
\newcommand{\ea}{\end{array}}
\newcommand{\bea}{\begin{eqnarray*}}
\newcommand{\eea}{\end{eqnarray*}}
\newcommand{\bean}{\begin{eqnarray}}
\newcommand{\eean}{\end{eqnarray}}

\newtheorem{theorem}{Theorem}[section]
\newtheorem{lemma}{Lemma}[section]

\newtheorem{corollary}{Corollary}[section]

\newcommand{\lc}{\mathrel{\raise2pt\hbox{${\mathop<\limits_{\raise1pt\hbox{\mbox{$\sim$}}}}$}}}
\newcommand{\gc}{\mathrel{\raise2pt\hbox{${\mathop>\limits_{\raise1pt\hbox{\mbox{$\sim$}}}}$}}}
\newcommand{\ec}{\mathrel{\raise1pt\hbox{${\mathop=\limits_{\raise2pt\hbox{\mbox{$\sim$}}}}$}}}

\begin{document}

\title{Numerical Energy Dissipation for Time-Fractional Phase-Field Equations}

\author[1]{Chaoyu Quan}
\author[2,4]{Tao Tang}
\author[3,4]{Jiang Yang}
\affil[1]{\small SUSTech International Center for Mathematics, Southern University of Science and Technology, Shenzhen, China (\href{mailto:quancy@sustech.edu.cn}{quancy@sustech.edu.cn}).}
\affil[2]{\small Division of Science and Technology, BNU-HKBU United International College, Zhuhai, Guangdong, China (\href{mailto:ttang@uic.edu.cn}{ttang@uic.edu.cn}).}
\affil[3]{\small Department of Mathematics, Southern University of Science and Technology, Shenzhen, China  (\href{mailto:yangj7@sustech.edu.cn}{yangj7@sustech.edu.cn}).}
\affil[4]{\small Guangdong Provincial Key Laboratory of Computational Science and Material Design, Southern University of Science and Technology, Shenzhen, China.}
\maketitle

\begin{abstract}
The numerical integration of phase-field equations is a delicate task which needs to
recover at the discrete level intrinsic properties of the solution such as energy dissipation and maximum principle.
Although the theory of energy dissipation for classical phase field models is well established, the corresponding theory for
 time-fractional phase-field models is still incomplete.
In this article, we study certain nonlocal-in-time energies using the first-order stabilized semi-implicit L1 scheme. In particular, we will
establish a discrete fractional energy law and a discrete weighted energy law. The extension for a ($2-\alpha$)-order L1 scalar auxiliary variable scheme will be investigated. Moreover, we demonstrate that the energy bound is preserved for the L1 schemes with nonuniform time steps.
Several numerical experiments are carried to verify our theoretical analysis.
\end{abstract}


{\bf AMS: }
65M06, 65M12, 74A50

\section{Introduction}
\setcounter{equation}{0}

A fractional time derivative arises when the characteristic waiting time diverges, which
models situations involving memory. In recent years, to model memory
effects and subdiffusive regimes in applications such as transport theory, viscoelasticity,
rheology and non-Markovian stochastic processes, there has been an increasing interest
in the study of time-fractional differential equations, i.e., differential equations where the
standard time derivative is replaced by a fractional one, typically a Caputo or a Riemann-Liouville derivative.
It has been reported that the presence of nonlocal operators in time
in the relevant governing equations may change diffusive dynamics significantly, which
can better describe certain fundamental relations between the processes of
interest, see, e.g.,
\cite{Caff17,PMedia2013,DCL2005,GroundW2017,naber2004time}.
It is also noted that an intensive effort has been put into investigations
on time fractional phase-field models. For instance, phase-field framework has been successfully employed to describe
the evolution of structural damage and fatigue \cite{caputo2015}, in which the damage is described by a variable order time fractional
derivative.

%It is noticed that there exist active theoretical research on time-fractional problems with spatial nonlinearity, which
%arises in practical applications. For example, Allen, Caffarelli and Vasseur \cite{Caff17} considered
%a time-space fractional porous medium equation with Caputo fractional time derivatives
%and nonlocal diffusion effects. In \cite{Giga17}, Giga and Namba investigated the well-posedness of
%Hamilton-Jacobi equations with a Caputo fractional time derivative, with a main purpose
%of finding a proper notion of viscosity solutions so that the underlying Hamilton-Jacobi
%equation is well-posed. A further study along this line is recently provided by Camilli and
%Goffi \cite{Camilli20}.
%Their study relies on a combination of a gradient bound for the time-fractional
%Hamilton-Jacobi equation obtained via nonlinear adjoint method and sharp estimates in
%Sobolev and Hölder spaces for the corresponding linear problem.

Seeking numerical solutions of phase field problems has attracted a lot of recent
attentions. The numerical integration of phase-field equations can be a delicate task: it needs to
recover at the discrete level intrinsic properties of the solution (energy diminishing, maximum principle)
and the presence of small parameter $\varepsilon > 0$ (typically, the interphase length) can generate practical difficulties.
Numerical analysis and computation aiming to handle this task for the classical phase field problems have attracted extensive
attentions, see, e.g., \cite{du2020review,hughes2011,shenxuyang19,tang2020A} and the references therein.
On the other hand, it is natural to extend the relevant discrete level intrinsic properties,
i.e., the maximum principle and energy stability to handle the
time-fractional phase-field equations, see, e.g.,  \cite{du2020time,lisalgado21,WangH2018,Luchko2017}.

This work is concerned with numerical methods for time-fractional phase-field models with the Caputo time-derivative.
The time-fractional phase-field equation can be written in the form of
\begin{subequations}\label{eq:phase_field}
\begin{align}
%\begin{equation}\label{eq:phase_field}
	 \partial^\alpha_t \phi =  \gamma\,\mathcal G \mu, \label{eq:phase_field1}
%\end{equation}
\end{align}
where $\alpha \in (0,1)$, $\gamma>0$ is the mobility constant, $\mathcal G$ is a nonpositive operator, and $\partial_t^\alpha$ is the Caputo fractional derivative
%\cite{caputo1967linear}
defined by
\begin{align}
%\begin{equation}
\partial^\alpha_t \phi(t) \coloneqq \frac{1}{\Gamma(1-\alpha)}\int_0^t \frac{\phi'(s)}{(t-s)^\alpha}\, {\rm d} s, \quad t\in(0,T),
\label{eq:phase_field2}
\end{align}
%\end{equation}
\end{subequations}
with $\Gamma(\cdot)$ the gamma function.
Choosing different $\mathcal G$ and $\mu$, one derives different phase-field models, such as the Allen--Cahn (AC) model and the Cahn--Hilliard (CH) model. In the AC model  and the CH  model, $\mathcal G $ is taken to be $-1$ and $\Delta$,  respectively, while in both cases $\mu$
takes the same form
\begin{equation}\label{eq:acch}
\mu = -\varepsilon^2 \Delta  \phi + F'(\phi),
\end{equation}
where $\varepsilon>0$ is the interface width parameter and $F$ is a double-well potential functional, commonly chosen as $F(\phi)  = \frac 1 4 \left(1-\phi^2\right)^2$ so that $ F'(\phi) = \phi^3-\phi$. Moreover, the MBE model has two forms, with or without slope selection \cite{xu2006stability}, where
\begin{subequations}\label{eq:mbe}
\begin{align}
\mathcal G = -1, \qquad \mu = \varepsilon^2 \Delta^2  \phi + \nabla\cdot {\mathbf f_{\rm m}} (\nabla \phi), \label{eq:mbe1}
\end{align}
with
\begin{align}
  \begin{array}{r@{}l}
 	\mathbf f_{\rm m} (\nabla \phi) = \left\{
	\begin{aligned}
  & \nabla\phi  - |\nabla\phi|^2 \nabla\phi, && \mbox{with slope selection},\\
 &  \frac{\nabla \phi}{1+\left|\nabla \phi\right|^2}, && \mbox{without slope selection}.
 	\end{aligned}
	\right.
  \end{array} \label{eq:mbe2}
  \end{align}
\end{subequations}
For sake of simplicity, we consider  the periodic boundary condition for above time-fractional phase-field problems.

The classical energy for the standard Allen--Cahn or Cahn--Hilliard equation (i.e., \eqref{eq:acch} with $\alpha=1$) is
\begin{equation}\label{eq:energy}
E(\phi) = \int_\Omega \left(\frac{\varepsilon^2} 2 \left| \nabla \phi \right|^2 + F(\phi) \right) \, {\rm d} x,
\end{equation}
while for the MBE equation \eqref{eq:mbe} is given by
\begin{subequations}\label{eq:energy_mbe}
\begin{align}
E(\phi) = \int_\Omega \left(\frac{\varepsilon^2} 2 \left| \Delta \phi \right|^2 + F_{\rm m}(\nabla\phi)\right) \, {\rm d} x,
\label{eq:energy_mbe1}	
\end{align}
with
\begin{align}
F_{\rm m}(\nabla\phi) = \left\{
  \begin{array}{r@{}l}
	\begin{aligned}
	&  \frac 1 4 \left(1- \left| \nabla\phi \right|^2\right)^2, && \mbox{with slope selection}, \\
	& - \frac 1 2 \ln\left(1 + \left| \nabla\phi \right|^2\right), && \mbox{without slope selection}.
	\end{aligned}
  \end{array}
  \right. \label{eq:energy_mbe2}
  \end{align}
\end{subequations}
Applying the definition (\ref{eq:energy}) with the time-fractional problem \eqref{eq:acch} gives
\begin{equation}
\frac{\rm d}{{\rm d} t} E(\phi)=  \frac 1 \gamma \int_\Omega \partial_t \phi  \left({\mathcal G^{-1}}\partial^\alpha_t \phi  \right) {\rm d}x,
\end{equation}
where $\mathcal G^{-1}$ is the inverse of $\mathcal G$.

It is well known that when $\alpha=1$, the Allen--Cahn and  Cahn--Hilliard  models are gradient flows: the energy associated with these models decays with respect to time, which is the so-called  energy dissipation law.
This property has been used extensively as a nonlinear numerical stability criterion.
However, it is still unknown if such energy dissipation property holds in the general case of $0<\alpha<1$.
In a recent work \cite{TangYZ19}, it is demonstrated that the classical energy \eqref{eq:phase_field} is bounded above by the initial energy:
\begin{equation}
E(t)\le E(0)\qquad \forall\; 0<t<T,
\end{equation}
which is the first work on the energy stability for time-fractional phase-field equations.
Later, Du et al. \cite{du2020time} proposed the fractional energy law based on numerical observations and only proved for the convex energy (not applicable to nonconvex phase-field models), i.e.,
the time-fractional derivative of energy is nonpositive:
\begin{equation} \label{eq:newenergy0}
\partial_t^\alpha E(t)\leq 0\quad \forall\; 0<t<T.
\end{equation}
In \cite{quan2020define}, this fractional energy law was proved for general cases. Still in \cite{quan2020define}, it is shown that in the continuous case a {\em weighted} energy decays with respect to time:
\begin{subequations} \label{eq:newenergy1}
\begin{align}
\partial_t E_\omega(t) \leq 0\qquad \forall\; 0<t<T, \label{eq:newenergy11}
\end{align}
where $E_\omega$ is defined by
\begin{align}
\label{eq:newenergy12}
	E_\omega(t)=  \frac{1}{B(\alpha,1-\alpha)}\int_0^t \frac{E(s)}{s^{1-\alpha}(t-s)^\alpha} \, {\rm d} s.
\end{align}
\end{subequations}
By using the transformation $s=\theta t$, we can obtain
\[
E_\omega(t)=  \frac{1}{B(\alpha,1-\alpha)}\int_0^1 \frac {E(\theta t)} {\theta ^{1-\alpha} (1- \theta)^\alpha} d\theta.
\]
Taking derivative with respect to $t$ and using the transformation $\theta = s/t$ yield 
\begin{equation}\label{eq:decayA}
E'_\omega(t) = \frac{1}{B(\alpha,1-\alpha) t}\int_0^t \frac{s^{\alpha}E'(s)}{(t-s)^\alpha} \, {\rm d} s.
\end{equation}
In other words, the fractional energy law of (\ref{eq:newenergy0}) and the weighted energy dissipation law (\ref{eq:newenergy1}) are all associated with the Caputo fractional operator, i.e., the dissipation of certain time-fractional form for energy:
\begin{equation} \label{explian}
\partial_t^\alpha E(t) \sim \int_0^t \frac{E'(s)}{(t-s)^\alpha}\, {\rm d} s\le 0,  \quad
E'_\omega(t) \sim \int_0^t \frac{s^{\alpha} E'(s)}{(t-s)^\alpha} \, {\rm d} s \le 0.
\end{equation}

This paper is concerned with the numerical implementation of the energy stability properties (\ref{eq:newenergy0}) and (\ref{eq:newenergy1}) by using the first-order stabilized L1 scheme. From a numerical point view, the essential step to study the fractional PDE is to approximate the time-fractional derivative operator. The classical L1 method is naturally derived from the approximation of the fractional integral as a Riemann sum and long known to be consistent (see, e.g., \cite{lisalgado21}).
Moreover, L1 approximation scheme stands out by being able to preserve at the
discrete level certain desirable features of the original PDEs, such as maximum principle \cite{giga2019discrete,jinlizhou18,TangYZ19}
and energy stability \cite{du2020time,huangstyne20}.
Our analysis will also be relevant to the convex-splitting schemes \cite{chen2012linear,eyre1998unconditionally,wang2010unconditionally}, the stabilization schemes  \cite{shen2010numerical,xu2006stability}, and the scalar auxiliary variable (SAV) schemes \cite{shen2018scalar}.
In particular, we will establish the energy boundedness and the fractional energy law for a ($2-\alpha$)-order L1-SAV scheme with uniform time steps. The energy boundedness under non-uniform time step will be also investigated.
All proofs are based on a special Cholesky decomposition proposed recently by us in \cite{quan2020define}, which seems very useful for studying numerical approximations of time-fractional phase-field equations.

We point out a very recent comprehensive study for the theory of fractional gradient flows by Li and Salgado \cite{lisalgado21} which
introduced the notion of energy solutions. The authors provide existence, uniqueness and certain regularizing effects due to the
Caputo derivative. The time-fractional phase-field models fit well with the class of the fractional gradient flows.

The paper is organized as follows.
Section \ref{sect2} introduces the L1 approximation of the time-fractional operator and a semi-implicit stabilization technique.
We then establish the fractional energy law (\ref{eq:newenergy0}) and the weighted energy dissipation law
(\ref{eq:newenergy1}).
In Section \ref{sect3}, we propose a ($2-\alpha$)-order L1-SAV scheme, and establish the corresponding energy boundedness and the fractional energy law.
In Section \ref{sect4}, the first-order and ($2-\alpha$)-order L1 schemes are investigated with nonuniform time-steps. Section \ref{sect5} presents several numerical examples to verify our theoretical results.
Some concluding remarks are given in the final section.

\section{First order stabilized L1 scheme} \label{sect2}
\setcounter{equation}{0}

We first introduce the discretization of the time fractional derivative.
Let $\Delta t = T /N$ be the time step size and $t_n = n \Delta t$, $0\leq n \leq N$.
The L1 approximation of the Caputo time-fractional derivative (\ref{eq:phase_field2}) is given by:
\begin{subequations}
\label{eq:partial_t_alpha}
\begin{align}
  \bar\partial_n^\alpha \phi \coloneqq \sum_{j=1}^{n} b_{n-j} \overline \partial_j \phi,\quad 1\leq n \leq N,
\label{eq:partial_t_alpha1}
\end{align}
where $\overline \partial_n^\alpha$ denotes the discrete fractional derivative at $t_n$,
\begin{align}
\label{eq:partial_t_alpha2}
b_j = \frac{\Delta t^{1-\alpha}}{ \Gamma(2-\alpha)}\left[ (j+1)^{1-\alpha} - j^{1-\alpha} \right],
\quad j\geq 0,
\end{align}
and $\overline \partial_k $ denotes the discrete first-order derivative at $t_k$ as follows:
\begin{align}
 \overline \partial_k \phi \coloneqq \frac{ \phi^{k} - \phi^{k-1} }{\Delta t}. \label{eq:partial_t_alpha3}
\end{align}
\end{subequations}
One can refer to \cite{sun2006fully} for the analysis of the  L1 approximation, where the truncation error of order $2-\alpha$
is derived. A useful reformulation of \eqref{eq:partial_t_alpha1} is
\begin{equation}
\bar\partial_n^\alpha \phi =  \frac{1}{\Delta t}\left[ b_0 \phi^{n} - \sum_{j=1}^{n-1} (b_{j-1}-b_{j})  \phi^{n-j}  -b_{n-1}\phi^0\right],\quad 1\leq n \leq N,
\end{equation}
where the following relationship holds:
\begin{equation}
b_{j-1}-b_{j}>0,\quad b_{n-1} >0,\quad \sum_{j=1}^{n-1} (b_{j-1}-b_{j}) + b_{n-1} = b_0.
\end{equation}
We further decompose the energy by the quadratic-nonquadratic splitting as follows:
\begin{equation}\label{eq:quadratic_splitting}
E(\phi) = \frac 1 2 \left<\phi,\mathcal L \phi\right> + E_1(\phi),
\end{equation}
where $\left<\cdot,\cdot\right>$ denotes the $\ell_2$ inner product over $\Omega$, $\mathcal L$ is some symmetric nonnegative linear operator ($\mathcal L = - \varepsilon^2 \Delta$ for the AC/CH model and $\mathcal L = \varepsilon^2\Delta^2$ for the MBE model), and $E_1$ is the remaining nonquadratic term.
The stabilized L1 scheme for \eqref{eq:phase_field} is written as
\begin{equation}\label{eq:L1STA}
\bar\partial_{n+1}^\alpha \phi = \gamma \,\mathcal G \left(\mathcal L \phi^{n+1} +  \delta_\phi E_{1}(\phi^n) + \widetilde{\mathcal L} \left(\phi^{n+1}-\phi^n\right) \right),
\end{equation}
where $\widetilde{\mathcal L}$ is some linear operator in the following form (see, e.g.,  \cite{TangYZ19})
\begin{equation}
\widetilde{\mathcal L} = \left\{
  \begin{array}{l@{}l}
	\begin{aligned}
	& \; S && \mbox{AC or CH model,}\\
	& -S\Delta  && \mbox{MBE without slope selection}
	 \end{aligned}
  \end{array}
  \right.
\end{equation}
with some positive constant $S$ satisfying
\begin{equation}
S \geq \left\{
  \begin{array}{r@{}l}
	\begin{aligned}
	& 2 && \mbox{AC model,}\\
	& \frac L 2 && \mbox{CH model,}\\
	& \frac 1 {16} && \mbox{MBE without slope selection.}
	 \end{aligned}
  \end{array}
  \right.
\end{equation}
Here, for the CH model, a truncation technique as in \cite{caffarelli95,shen2010numerical} is used so that
$$\max_{\phi\in \mathbb R} \left| \left(\delta_\phi E_1(\phi)\right)'\right|\leq L$$
 for some constant $L>0$.

%\item[(3)] The SAV scheme for \eqref{eq:phase_field} is written generally as
%\begin{equation}\label{eq:scheme_sav}
%  \begin{array}{r@{}l}
%	\begin{aligned}
%  \overline\partial_{n+1}^\alpha \phi & =  \gamma \, {\mathcal G} \mu^{n+1},\\
% \mu^{n+1} & = \mathcal L \phi^{n+1} + \frac{r^{n+1}}{\sqrt{E_1(\phi^n)+C_0}} \delta_\phi E_{1}(\phi^n), \\
% r^{n+1}-r^n & = \frac{1}{2\sqrt{E_1(\phi^n)+C_0}} \left< \delta_\phi E_{1}(\phi^n),\phi^{n+1}-\phi^n\right>,
% 	\end{aligned}
%  \end{array}
%\end{equation}
%where $\mathcal L$ is a symmetric nonnegative linear operator, $E_1(\phi)$ is given by \eqref{eq:quadratic_splitting}, { $C_0>0$ is some positive constant, and $\left<\cdot,\cdot\right>$ denotes the $\ell_2$ inner product over $\Omega$}.

In the classical case of $\alpha = 1$, the following inequality holds for the scheme \eqref{eq:L1STA}:
\begin{equation}
E^{n+1} - E^n \leq \frac 1 {\gamma\Delta t} \left<  \mathcal G^{-1} (\phi^{n+1} - \phi^n),  \phi^{n+1} - \phi^n \right> \leq 0,
\end{equation}
since $\mathcal G^{-1}$ is nonpositive definite.
Similarly, in the general case of $0<\alpha <1$, one can obtain the following inequality characterizing the energy difference between two neighboring time steps.

\vskip .25cm
\begin{lemma}\label{lem:diffene}
Assume that the initial data satisfies $\Vert \phi^0 \Vert _\infty \le 1$. The energy of the stabilized L1 scheme \eqref{eq:L1STA} satisfies the following property:
\begin{equation}\label{ineq:diffene}
E^{n+1} - E^n \leq \frac 1 \gamma \left<  \mathcal G^{-1} \overline \partial_{n+1}^\alpha \phi,  \phi^{n+1} - \phi^n \right>, \quad 0\leq n\leq N-1,
\end{equation}
where $E^n = E(\phi^n)$ denotes the classical energy at $t_n$.
\end{lemma}
\begin{proof}
The proof is quite similar to the classical case of $\alpha = 1$ given by \cite{TangYZ19}.
Here, we only prove the specific case of AC model under the constraint $S\geq 2$.  In this case, \eqref{eq:L1STA} can be rewritten as
\begin{eqnarray}
&& \left(\frac{b_0}{\Delta t} + \gamma S -\gamma \varepsilon^2 \Delta\right) \phi^{n+1} \nonumber \\
&=& \gamma ( S+1) \phi^n - \gamma(\phi^n)^3 + \sum_{j=0}^{n-1} \frac{(b_j-b_{j+1})}{\Delta t}  \phi^{n-j}  + \frac{b_n}{\Delta t}\phi^0.
\end{eqnarray}
Since $S\geq 2$, it is not difficult to verify that if $\|\phi^n\|_\infty\leq 1$ then $\|( S+1) \phi^n - (\phi^n)^3\|_\infty \leq  S$.
Further, it is known (see, e.g., \cite{tang2016implicit}) that
\begin{equation}
\left\| \left(\frac{b_0}{\Delta t} + \gamma S -\gamma \varepsilon^2 \Delta\right)^{-1}\right\|_\infty \leq  \left(\frac{b_0}{\Delta t} + \gamma S\right)^{-1},
\end{equation}
which yields by induction that if $\left \| \phi^0 \right \|_\infty \leq 1$ then $\left \| \phi^{n} \right \|_\infty \leq 1$, i.e., the maximum bound is preserved. As a consequence, we further obtain
  \begin{eqnarray}
	&&\frac 1 \gamma \left<  \mathcal G^{-1} \overline \partial_{n+1}^\alpha \phi,  \phi^{n+1} - \phi^n \right>  \nonumber \\
	&=& \left< - \varepsilon^2 \Delta \phi^{n+1} +  (\phi^n)^3 - \phi^n +  S  \left(\phi^{n+1}-\phi^n\right), { \phi^{n+1} - \phi^n } \right> \nonumber \\
	&=& \frac {\varepsilon^2}{2} \left( \left\|\nabla \phi^{n+1}\right\|^2 - \left\|\nabla \phi^n\right\|^2 + \left\|\nabla \phi^{n+1} - \nabla \phi^n\right\|^2 \right) \nonumber \\
	&&\quad +   \left< (\phi^n)^3 - \phi^n +  S  \left(\phi^{n+1}-\phi^n\right) , \phi^{n+1} - \phi^n \right> \nonumber\\
	& \geq& \frac {\varepsilon^2}{2} \left( \left\|\nabla \phi^{n+1}\right\|^2 - \left\|\nabla \phi^n\right\|^2  \right)
	 +  \frac{1}{4} \left( \left\| (\phi^{n+1})^2 - 1 \right\|^2  - \left\| (\phi^n)^2 - 1 \right\|^2  \right) \nonumber\\
	&=&  E^{n+1} - E^n,
 \end{eqnarray}
where the following inequality is used:
\begin{equation*}
(b^3-b)(a-b) + (a-b)^2 \geq \frac 1 4 \left[ (a^2-1)^2 - (b^2-1)^2 \right]\quad \forall\; a,b\in [-1,1].
\end{equation*}
This completes the proof of the lemma.
\end{proof}

\vskip .3cm
We point out that when $\alpha = 1$ \eqref{ineq:diffene} indicates that the discrete energy $E^n$ decays w.r.t. $n$.
When $0<\alpha<1$, \eqref{ineq:diffene} will be useful in our later analysis for the fractional energy law and the weighted energy dissipation law.

Note that the first-order convex-splitting scheme for \eqref{eq:phase_field} can be written as
\begin{equation}
\bar\partial_{n+1}^\alpha \phi = \gamma \, \mathcal G \left(\delta_\phi E_{\rm c}(\phi^{n+1})-\delta_\phi E_{\rm e}(\phi^n)\right),
\end{equation}
where $\bar \partial_n^\alpha$ is given by \eqref{eq:partial_t_alpha1}, $E_{\rm c}$ and $E_{\rm e}$ are two convex functionals w.r.t. $\phi$ such that $E(\phi) = E_{\rm c}(\phi) - E_{\rm e}(\phi)$.
The first-order SAV scheme for \eqref{eq:phase_field} is given by (see, e.g., \cite{shenxuyang19}):
\begin{subequations} \label{2e17}
\begin{align}
 & \bar \partial_{n+1}^\alpha \phi  =  \gamma \, {\mathcal G} \mu^{n+1},\label{2e171}\\
 &\mu^{n+1} = \mathcal L \phi^{n+1} + \frac{r^{n+1}}{\sqrt{E_1(\phi^n)+C_0}} \delta_\phi E_{1}(\phi^n), \label{2e17b}\\
 & r^{n+1}-r^n = \frac{1}{2\sqrt{E_1(\phi^n)+C_0}} \left< \delta_\phi E_{1}(\phi^n),\phi^{n+1}-\phi^n\right>, \label{2e17c}
 	\end{align}
\end{subequations}
where $C_0>0$ is some positive constant.
It can be verified without difficulty that both schemes still satisfy \eqref{ineq:diffene}, implying that the fractional and weighted energy laws in subsections \ref{subsect2.1} and \ref{subsect2.2} also hold.

\vskip .3cm
Before studying the discrete energy laws, we first recall a special Cholesky decomposition result which provides a new way to determine positive definiteness of a symmetric positive matrix.

\vskip .3cm
\begin{lemma}[A special Cholesky decomposition, \cite{quan2020define}]\label{lem:cholesky}
Given an arbitrary symmetric matrix $\mathbf S$ of size $N\times N$ with positive elements. If $\mathbf S$ satisfies the following properties:
\begin{itemize}\label{item:1}
\item[(P1)] $\forall\; 1\leq j < i \leq N$, $\mathbf S_{i-1,j}\geq  \mathbf S_{i, j}$;
\item[(P2)] $\forall\; 1 < j \leq i \leq N$, $\mathbf S_{i, j-1}<  \mathbf S _{i, j}$;
\item[(P3)] $\forall\; 1< j < i \leq N$, $\mathbf S_{i-1, j-1} -\mathbf S_{i, j-1}\leq \mathbf S_{i-1, j} -\mathbf S_{i, j}$,
\end{itemize}
then ${\mathbf S}$ is a positive definite matrix. Moreover, ${\mathbf S}$ has a Cholesky decomposition
$\mathbf S = {\mathbf L} {\mathbf L}^{\rm T}$,
where ${\mathbf L}$ is a lower triangular matrix satisfying
\begin{itemize}
\item[(Q1)] $\forall 1\leq j\leq i\leq N$, $\left[ \mathbf L \right]_{ij} > 0$;
\item[(Q2)] $\forall 1\leq j < i \leq N$, $\left[ \mathbf L \right]_{i-1,j}\geq \left[ \mathbf L \right]_{i, j}$.
\end{itemize}
\end{lemma}

\vskip .3cm
 Note that the property (P1) indicates that the matrix $\mathbf S$ is column decreasing, while (P2) means that $\mathbf S$ is row increasing.
The property (P3) is related to the second-order cross partial derivative from the continuous point of view, see \cite{quan2020define} for more details.

\subsection{Fractional energy law}\label{subsect2.1}
We state and prove the first discrete energy law, called the fractional energy law, proposed by Du et al. \cite{du2020time}.
Our proof is based on the following lemma.

\vskip .25cm
\begin{lemma}\label{lem:L1prop1}
For any function $u$ defined on $\Omega \times [0,T]$, the following inequality holds:
\begin{equation}
\sum_{k=1}^{n} b_{n-k}  \left< \bar\partial_{k}^\alpha u, \bar\partial_{k} u\right> \geq 0 \quad \forall n\geq 1.
\end{equation}
\end{lemma}

\begin{proof}
It is sufficient to prove that
\begin{equation}
\mathbf B =
\left[\begin{array}{ccccc}b_{n-1} &  &  &  &  \\ & b_{n-2} &  &  &  \\ &  & \ddots &  &  \\ &  &  & b_1 &  \\ &  &  &  & b_0\end{array}\right]
\left[\begin{array}{ccccc}b_0 &  &  &  &  \\b_1 & b_0 &  &  &  \\\vdots & \vdots & \ddots &  &  \\b_{n-2} & b_{n-3} & \cdots & b_0 &  \\b_{n-1} & b_{n-2} & \cdots & b_1 & b_0\end{array}\right]
\end{equation}
is positive definite, which is equivalent to prove that $\mathbf B + \mathbf B^{\rm T}$ is positive definite.
To do this, we make a conjugate transformation of $\mathbf B + \mathbf B^{\rm T}$ as follows:
\begin{equation}
 {\mathbf S} = P \left({\mathbf B}+{\mathbf B}^{\rm T}\right) P^{\rm T},
\end{equation}
where $P$ is an anti-diagonal matrix given by
\begin{equation}\label{eq:matP}
P = \left[ \begin{array}{cccc} &  &  & b_{0}^{-1} \\ &  & b_{1}^{-1} &
\\ & \adots &  &  \\ b_{n-1}^{-1} &  &  & \end{array}\right]_{n\times n}.
\end{equation}
It is easy to verify that $\mathbf S$ can be written explicitly as
\begin{equation}
\mathbf S_{ij}  = \left\{
  \begin{array}{r@{}l}
	\begin{aligned}
	& 2 b_{0}b_{i-1}^{-1} && \mbox{if } i = j,\\
	& b_{i-j}b_{i-1}^{-1}  && \mbox{if } i > j,\\
	& b_{j-i}b_{j-1}^{-1}  && \mbox{if } i < j.
	 \end{aligned}
  \end{array}
  \right.
\end{equation}
%It is sufficient to show that $\mathbf S$ is positive definite.
With the above definition, it is not difficult to check that $\mathbf S$ satisfies (P1) and (P2) of Lemma \ref{lem:cholesky}.
We now prove that $\mathbf S$ also satisfies (P3).
In the case of $j =i-1$,  it is trivial to see that the property (P3) indeed holds.
In the general case of $1< j<i-1$ ($i\geq 4$), we need to show that
\begin{equation}
\mathbf S_{i-1,j-1} - \mathbf S_{i,j-1} \leq \mathbf S_{i-1,j} - \mathbf S_{i,j},
\end{equation}
that is equivalent to
%\begin{equation}\label{ineq:fj}
$f(j-1) \leq f(j)$,
%\end{equation}
where
%\begin{equation}\label{eq:funf}
\[
f(x) = \frac{(i-x)^{1-\alpha}-(i-x-1)^{1-\alpha}}{(i-1)^{1-\alpha}-(i-2)^{1-\alpha}}
-\frac{(i-x+1)^{1-\alpha}-(i-x)^{1-\alpha}}{i^{1-\alpha}-(i-1)^{1-\alpha}}.
\]
Note that $f(1) = 0$. Then $f'(x)\ge 0$ for $1<x \le i-1$ can ensure the desired inequality. Below we will verify the positivity of $f'(x)$. By direct computation we can that $f'(x)\ge 0$ is equivalent to, for $1<x \le i-1$,
\begin{equation}\label{eq:ineq_inc}
 \frac{1-(i-x-1)^{\alpha}(i-x)^{-\alpha}}{(i-x-1)^{\alpha} ((i-1)^{1-\alpha}-(i-2)^{1-\alpha})}
\geq  \frac{1-(i-x)^{\alpha}(i-x+1)^{-\alpha}}{(i-x)^{\alpha} (i^{1-\alpha}-(i-1)^{1-\alpha})}.
\end{equation}
It is not difficult to see the order of the numerators above:
%\begin{equation}\label{eq:ineq_inc0}
\[
1-(i-x-1)^{\alpha}(i-x)^{-\alpha} \geq {1-(i-x)^{\alpha}(i-x+1)^{-\alpha}}.
\]
To prove \eqref{eq:ineq_inc}, it is now sufficient to show that
\begin{equation}\label{eq:ineq_inc3}
(i-x-1)^\alpha((i-1)^{1-\alpha}-(i-2)^{1-\alpha}) \leq (i-x)^\alpha(i^{1-\alpha}-(i-1)^{1-\alpha}).
\end{equation}
To show this, we consider the auxiliary function $p(y) = y^\alpha ( (y+1)^{1-\alpha}-y^{1-\alpha})$ for  $y\in [0, \infty)$.
It is easy to verify that
\begin{eqnarray*}
p'(y) & = & \alpha y^{-(1-\alpha)} (y+1)^{1-\alpha} +  (1-\alpha) y^{\alpha} (y+1)^{-\alpha} - 1 \\
& =& \alpha z^{1-\alpha} +  (1-\alpha) z^{-\alpha} - 1\geq 0,
\end{eqnarray*}
where $z = \frac{y+1}{y} > 1$.
Therefore, we can obtain that $p(i-2)\leq p(i-1)$, i.e.,
\begin{equation}\label{eq:ineq_inc2}
(i-2)^\alpha((i-1)^{1-\alpha}-(i-2)^{1-\alpha}) \leq (i-1)^\alpha(i^{1-\alpha}-(i-1)^{1-\alpha}),\quad  i\geq 4.
\end{equation}
By multiplying \eqref{eq:ineq_inc2} with the following obvious inequality
%\begin{equation}\label{eq:ineq_inc1}
\[
\left(1-\frac{x-1}{i-2}\right)^\alpha \leq \left(1-\frac{x-1}{i-1}\right)^\alpha,
\]
we obtain \eqref{eq:ineq_inc3} and then \eqref{eq:ineq_inc}. Therefore, $f(x)$ is monotonically increasing.

In summary, $\mathbf S$ satisfies (P1)--(P3) in Lemma \ref{lem:cholesky}.
We then claim that $\mathbf S$ is positive definite and consequently, $\mathbf B$ is  positive definite.
\end{proof}

\vskip .25cm
\begin{theorem}[Fractional energy law]\label{main-thm1}
For the stabilized L1 scheme \eqref{eq:L1STA} to the  time-fractional phase-field equations, the following fractional energy law holds:
\begin{equation}\label{ineq:discrete_energy_dissipation1}
\overline \partial_n^\alpha E = \sum_{k=1}^{n} b_{n-k} \overline \partial_k E   \leq 0 \quad \forall\; 1\leq n \leq N,
\end{equation}
where the discrete fractional derivative $\overline \partial_n^\alpha$ is given by \eqref{eq:partial_t_alpha1}, but now acts on $E^n$.
\end{theorem}

\begin{proof}
It follows from Lemma \ref{lem:diffene} and the definition of discrete fractional derivative that
\begin{equation}\label{eq:discrete_ED_3.1}
  \begin{array}{r@{}l}
	\begin{aligned}
	\overline \partial_n^\alpha E
	& \leq \frac{1}{\gamma} \sum_{k=1}^{n}  b_{n-k}  \left<  \mathcal G^{-1} \overline \partial_k^\alpha \phi,  \overline \partial_k \phi\right> = - \frac{1}{\gamma} \sum_{k=1}^{n}  b_{n-k} \left<  \overline \partial_k^\alpha \psi,  \overline \partial_k \psi\right>,
 	\end{aligned}
  \end{array}
\end{equation}
where $\forall\; 1\leq k\leq n$,
\begin{equation}\label{eq:psik}
  \psi^k = \left\{
  \begin{array}{r@{}l}
	\begin{aligned}
	&  \overline \partial_k \phi && \mbox{Allen--Cahn or MBE model}, \\
	&\nabla (-\Delta)^{-1} \overline \partial_k \phi  && \mbox{Cahn--Hilliard model}.
 	\end{aligned}
  \end{array}
  \right.
\end{equation}
According to Lemma \ref{lem:L1prop1}, we then have $\overline \partial_n^\alpha E \leq 0$.
\end{proof}

\vskip .25cm
We point out that the energy boundedness obtained in \cite{TangYZ19} is a direct corollary of
Theorem \ref{main-thm1}.

\vskip .3cm
\begin{corollary}[Energy boundedness]\label{colloary}
For the L1 scheme \eqref{eq:L1STA} of the time-fractional phase-field models, the discrete energy at $t_n$ is bounded above by the initial energy:
\begin{equation}
E^{n} \leq E^0\quad \forall\; 1\leq n \leq N.
\end{equation}
\end{corollary}

\begin{proof}
It follows from \eqref{ineq:discrete_energy_dissipation1} in Theorem \ref{main-thm1} that
\begin{equation}
E^n \leq  \frac{1}{b_0}\sum_{k=1}^{n-1} (b_{n-k-1}-b_{n-k})  E^{k}  + \frac{b_{n-1}}{b_0} E^0.
\end{equation}
When $n = 1$, this inequality gives $E^1\leq E^0$.
By induction on $n$, it is easy to see that $E^n\leq E^0$ always holds.
\end{proof}

\subsection{Weighted energy dissipation law}
\label{subsect2.2}

In \cite{quan2020define},  a weighted energy $E_\omega(t)$ is proposed for time-fractional phase-field equations in the form of
\begin{equation}
\label{eq:newenergy}
	E_\omega(t)=  \frac{1}{B(\alpha,1-\alpha)}\int_0^t \frac{E(s)}{s^{1-\alpha}(t-s)^\alpha} \, {\rm d} s,
\end{equation}
where $B(\alpha,1-\alpha)$ is the Beta function.
It is proved that this weighted energy decays with time on the continuous level, i.e.,
\begin{equation}\label{eq:decay}
 %E'_\omega(t) = \frac{1}{B(\alpha,1-\alpha) t}\int_0^t \frac{s^\alpha}{(t-s)^\alpha} E'(s) \, {\rm d} s  \leq 0.
E'_\omega(t) \le 0.
\end{equation}
Before stating the discrete weighted energy law, we provide a useful lemma.
%propose the following lemma on the property of L1 approximation.

\begin{lemma}\label{lem:L1prop2}
For any function $u$ defined on $\Omega \times [0, T]$, the following inequality holds:
\begin{equation}
\sum_{k=1}^{n} t_k^\alpha\, b_{n-k}  \left< \bar\partial_{k}^\alpha u, \bar\partial_{k} u\right> \geq 0 \quad \forall n\geq 1.
\end{equation}
\end{lemma}

\begin{proof}
It is sufficient to prove that
\begin{equation*}
\mathbf B =
\left[\begin{array}{ccccc} 1^\alpha b_{n-1}  &  &  &  &  \\ & 2^\alpha b_{n-2} &  &  &  \\ &  & \ddots &  &  \\ &  &  & (n-1)^\alpha b_1 &  \\ &  &  &  & n^\alpha b_0\end{array}\right]
\left[\begin{array}{ccccc}b_0 &  &  &  &  \\b_1 & b_0 &  &  &  \\\vdots & \vdots & \ddots &  &  \\b_{n-2} & b_{n-3} & \cdots & b_0 &  \\b_{n-1} & b_{n-2} & \cdots & b_1 & b_0\end{array}\right]
\end{equation*}
is positive definite.
It is sufficient to show that the symmetric matrix $\mathbf B + \mathbf B^{\rm T}$ is positive definite.
Similar to the proof of Lemma \ref{lem:L1prop1}, we make the following conjugate transformation:
${\mathbf S} = P \left({\mathbf B}+{\mathbf B}^{\rm T}\right) P^{\rm T}$,
where the anti-diagonal matrix $P$ is given by \eqref{eq:matP}.
Then, $\mathbf S$ can be written explicitly as
\begin{equation}
\mathbf S_{ij}  = \left\{
  \begin{array}{r@{}l}
	\begin{aligned}
	& 2 (n-i+1)^\alpha b_{0}b_{i-1}^{-1} && \mbox{if } i = j,\\
	& (n-j+1)^\alpha b_{i-j}b_{i-1}^{-1}  && \mbox{if } i > j,\\
	& (n-i+1)^\alpha b_{j-i}b_{j-1}^{-1}  && \mbox{if } i < j.
	 \end{aligned}
  \end{array}
  \right.
\end{equation}
To prove the positive definiteness of $\mathbf S$, we need to show that $\mathbf S$ satisfies the three properties (P1)--(P3) in Lemma \ref{lem:cholesky}.

We first check (P1). In fact, it is sufficient to show that for any fixed $j$, the following inequality holds for all $i\geq j\geq 1$:
\begin{eqnarray}
 && \frac{b_{i-j} }{b_{i-1}} = \frac{(i-j+1)^{1-\alpha}-(i-j)^{1-\alpha}}{i^{1-\alpha}-(i-1)^{1-\alpha}} \nonumber \\
 &\geq& \frac{b_{i-j+1} }{b_{i}} = \frac{(i-j+2)^{1-\alpha}-(i-j+1)^{1-\alpha}}{{(i+1)}^{1-\alpha}-i^{1-\alpha}}, \nonumber
\end{eqnarray}
which is equivalent to
\begin{equation}\label{eq:ineqP1}
\frac{{(i+1)}^{1-\alpha}-i^{1-\alpha}}{i^{1-\alpha}-(i-1)^{1-\alpha}}
\geq
\frac{(i-j+2)^{1-\alpha}-(i-j+1)^{1-\alpha}}{(i-j+1)^{1-\alpha}-(i-j)^{1-\alpha}}.
\end{equation}
We consider the following function
\begin{equation*}
f(x) = \frac{{(x+1)}^{1-\alpha}-x^{1-\alpha}}{x^{1-\alpha}-(x-1)^{1-\alpha}},\qquad x\geq 1.
\end{equation*}
It is easy to verify that
\[
f'(x) = \frac{ (1-\alpha)\left( 2x^\alpha - (x-1)^\alpha - (x+1)^\alpha \right) }{x^\alpha(x-1)^\alpha(x+1)^\alpha(x^{1-\alpha}-(x-1)^{1-\alpha})^2}
\geq 0.
\]
Since $j\geq 1$, we can then claim that $f(i) \geq f(i-j+1)$. Hence \eqref{eq:ineqP1} is true.
Consequently, $\mathbf S$ satisfies the property (P1).

We then check (P2). We shall prove that for any fixed $i$, the following inequality holds for $1\leq j< i\leq n$:
\[
(n-j+1)^\alpha b_{i-j} \leq (n-j)^\alpha b_{i-j-1},
\]
which is equivalent to
\begin{eqnarray}
&& (n-j+1)^\alpha \left((i-j+1)^{1-\alpha} - (i-j)^{1-\alpha} \right) \nonumber \\
&\leq& (n-j)^\alpha \left((i-j)^{1-\alpha}  - (i-j-1)^{1-\alpha} \right). \nonumber
\end{eqnarray}
Consider the auxiliary function
%\begin{equation}\label{eq:fung}
\[
g(x) = (n-x+1)^\alpha \left((i-x+1)^{1-\alpha}  - (i-x)^{1-\alpha} \right),\quad 1\leq x<i.
\]
Below we show that $g'(x)\geq 0$. Straightforward computation gives
\begin{equation*}
	g'(x)= -\frac{(1-\alpha)n + \alpha i -x+1  }{(n-x+1)^{1-\alpha} (i-x+1)^{\alpha}}
	+ \frac{(1-\alpha)n + \alpha i -x+1-\alpha }{(n-x+1)^{1-\alpha} (i-x)^{\alpha}}.
\end{equation*}
For any $1\leq x< i$, $g'(x)\geq 0$ is equivalent to
\begin{equation}\label{ineq:g_der_x}
\left( \frac{i-x+1}{i-x} \right)^\alpha \geq
\frac{(1-\alpha)n + \alpha i -x+1 }{(1-\alpha)n + \alpha i -x+1-\alpha}.
\end{equation}
Since $i\leq n$, the right-hand side of (\ref{ineq:g_der_x}) satisfies, for $1\le x < i$,
\begin{equation*}
\frac{(1-\alpha)n + \alpha i -x+1 }{(1-\alpha)n + \alpha i -x+1-\alpha} \leq \frac{ i -x+1 }{i -x+1-\alpha}.
\end{equation*}
In order to obtain \eqref{ineq:g_der_x}, it is sufficient to show the following inequality:
\begin{equation*}
\left( \frac{i-x+1}{i-x} \right)^\alpha \geq \frac{ i -x+1 }{i -x+1-\alpha},
\end{equation*}
that is,
\begin{equation}\label{ineq:g_der_x2}
i-x+1-\alpha \geq (i-x+1) \left(1-\frac{1}{i-x+1}\right)^\alpha.
\end{equation}
It can be easily verify by using Taylor expansion that
\begin{equation}
\left(1-\frac{1}{i-x+1}\right)^\alpha \leq 1-\frac{\alpha}{i-x+1},
\end{equation}
which yields \eqref{ineq:g_der_x2}. Consequently, \eqref{ineq:g_der_x} is true. Therefore, $g'(x)\ge 0$ which verifies (P2).

Thirdly, we check (P3). In the case of $j =i-1$,  it is trivial to show that the property (P3) holds according to (P1) and (P2).
In the general case of $2\leq j\leq i-2$ (hence, $4\leq i\leq n$), we shall prove
\begin{equation}
\mathbf S_{i-1,j-1} - \mathbf S_{i,j-1} \leq \mathbf S_{i-1,j} - \mathbf S_{i,j},
\end{equation}
which is equivalent to $h(j-1) \leq h(j)$,
where
%\begin{equation}\label{eq:funh}
\[
h(x) = (n-x+1)^\alpha \left[ \frac{(i-x)^{1-\alpha}-(i-x-1)^{1-\alpha}}{(i-1)^{1-\alpha}-(i-2)^{1-\alpha}}
-\frac{(i-x+1)^{1-\alpha}-(i-x)^{1-\alpha}}{i^{1-\alpha}-(i-1)^{1-\alpha}} \right],
\]
with $2\leq x \leq i-2$. Similar to the proof of positivity of $f'(x)$ and $g'(x)$, we can prove that $h'(x)\geq 0$.
Therefore, the property (P3) holds for $\mathbf S$.

In summary, $\mathbf S$ satisfies (P1)--(P3) in Lemma \ref{lem:cholesky}.
Consequently, $\mathbf S$ is positive definite and therefore $\mathbf B $ is positive definite.
\end{proof}

\vskip .25cm

\begin{theorem}[Weighted energy dissipation]\label{main-thm2}
For any $\alpha\in (0,1)$, the energy of the stabilized L1 scheme \eqref{eq:L1STA} for the time-fractional phase-field equations satisfies
\begin{equation}\label{eq:Ew_decay}
\widetilde E^n \leq \widetilde E^{n-1} \quad\forall\; 1\leq n\leq N,
\end{equation}
where $\widetilde E^n$ denotes the following discrete weighted energy:
\begin{equation}
\widetilde E^n = E^0 + \Delta t \sum_{m=1}^n D^m,
\quad\mbox{with}\quad
D^m = \frac{1}{\Gamma(\alpha) t_m}\sum_{k=1}^{m} t_k^\alpha \, b_{m-k} \, \overline \partial_k E.
\end{equation}
\end{theorem}

\begin{proof}
Note that $D^n$ is an approximation to the derivative of weighted energy in \eqref{eq:decay}.
To derive \eqref{eq:Ew_decay}, it is sufficient to prove $D^n\leq 0$ for all $1\leq n\leq N$.
According to \eqref{ineq:diffene}, we have the following inequality:
\begin{eqnarray}\label{eq:discrete_ED_3.2}
	\Gamma(\alpha)t_n D^n
	& =&  \sum_{k=1}^{n} t_{k}^\alpha b_{n-k} \overline \partial_k E
	 \leq \frac{1}{\gamma} \sum_{k=1}^{n} t_{k}^\alpha  b_{n-k}  \left<  \mathcal G^{-1} \overline \partial_k^\alpha \phi, \overline \partial_k \phi \right> \nonumber \\
	 & =& - \frac{1}{\gamma} \sum_{k=1}^{n} t_{k}^\alpha  b_{n-k}  \left<  \overline \partial_k^\alpha \psi, \overline \partial_k \psi \right>,
\end{eqnarray}
where $\psi^k$ is given by \eqref{eq:psik}.
Combining Lemma \ref{lem:L1prop2} and \eqref{eq:discrete_ED_3.2}, we conclude that $D^n\leq 0$.
\end{proof}

\vskip .25cm

We point out that Corollary \ref{colloary} can also be deduced from Theorem \ref{main-thm2}.
Further, the weighted energy dissipation law is even stronger than the fractional energy law result.
In fact, Theorem \ref{main-thm1} and Theorem \ref{main-thm2} state the following two inequalities respectively:
\begin{equation}
b_0 (E^n-E^0) \leq \sum_{k=1}^{n-1} (b_{n-k-1}-b_{n-k})\left(E^{k}-E^0\right), \label{ineq:rmk1}
\end{equation}
and
\begin{equation}
b_0 n^\alpha (E^n-E^0) \leq \sum_{k=1}^{n-1} \left[(k+1)^\alpha b_{n-k-1}-k^\alpha b_{n-k} \right]\left(E^{k}-E^0\right),
\label{ineq:rmk2}
\end{equation}
respectively. We can show that \eqref{ineq:rmk1} can be deduced from \eqref{ineq:rmk2}.
This proof is technical and is omitted here.

{
\section{($2-\alpha$)-order L1-SAV scheme}\label{sect3}
We have provided two discrete energy laws for the first-order stabilized L1 scheme, corresponding to the energy property of the governing equation introduced in \cite{du2020time,quan2020define}.
We will show in this section that Lemma \ref{lem:cholesky} can also be used to analyze the energy stability of high order scheme, i.e., the energy boundedness and the fractional energy law of a ($2-\alpha$)-order L1-SAV scheme.

Inspired by the extended SAV scheme in \cite{hou2019highly}, we consider a semi-discrete implicit scheme using the L1 approximation for the fractional derivative, the Crank--Nicolson discretization of the Laplace term, and the SAV technique \cite{shenxuyang19} for the nonlinear term:
\begin{subequations}
\label{eq:L1SAV}
\begin{align}
& \bar\partial_{n+\frac 1 2} ^\alpha \phi = \gamma \mathcal G \mu^{n+\frac 1 2},  \label{eq:L1SAV1}\\
& \mu^{n+\frac 1 2}  = \mathcal L \phi^{n+\frac 1 2} + \frac{r^{n+\frac 1 2}}{\sqrt{E_1(\bar \phi^{n+\frac 1 2})+C_0}} \delta_\phi E_1(\bar \phi^{n+\frac 1 2}), \label{eq:L1SAV2}\\
& r^{n+1}-r^{n}  = \frac{1}{2\sqrt{E_1(\bar \phi^{n+\frac 1 2})+C_0}} \left<\delta_\phi E_1(\bar \phi^{n+\frac 1 2}), \phi^{n+1}-\phi^n\right>. \label{eq:L1SAV3}
\end{align}
\end{subequations}
with
\[
 \phi^{n+\frac 1 2} =\frac 1 2(\phi^{n+1}+\phi^n),\quad r^{n+\frac 1 2} =\frac 1 2(r^{n+1}+r^n), \quad  \bar \phi^{n+\frac 1 2} = \frac 3 2 \phi^n-\frac 1 2 \phi^{n-1}.
\]
The discrete fractional derivative operator $\bar\partial_{n+\frac 1 2} ^\alpha $ is given by (see \cite{hou2019highly})
\begin{equation}\label{eq:L1half}
\bar\partial_{n+\frac 1 2}^\alpha  \phi = \sum_{j=0}^n \tilde b_{n-j} \bar\partial_{j+1}  \phi,
\end{equation}
where $\tilde b_0 = {\Delta t^{1-\alpha} 2^{\alpha -1}}/{ \Gamma(2-\alpha)}$ and
\[
\tilde b_j = \frac{\Delta t^{1-\alpha}}{ \Gamma(2-\alpha)}\left[ \left(j+\frac 1 2\right)^{1-\alpha} - \left(j-\frac 1 2\right)^{1-\alpha} \right],
\quad j \geq 1.
\]
 It is shown in \cite{hou2019highly}  that the order of truncation error of $\bar \partial^\alpha_{n+\frac 1 2} \phi$ to $\partial_t^\alpha \phi(t_{n+\frac 1 2})$ is $2-\alpha$. Hence the scheme \eqref{eq:L1SAV} is ($2-\alpha$)-order in time.

We now state two properties of the operator $\bar\partial_{n+\frac 1 2}^\alpha$, which ensures the energy boundedness and the fractional energy law.

\vskip .25cm

\begin{lemma}
For any function $u$ defined on $\Omega \times [0, T]$, the following inequalities hold:
\begin{eqnarray}
&& \label{eq:L1halfbound}
\sum_{k=0}^{n} \left<\bar\partial_{k+\frac 1 2}^\alpha u, \bar\partial_{k+1} u\right> \geq 0 \quad \forall n\geq 0,
\\
&& \label{eq:L1halfprop}
\sum_{k=0}^{n} \widetilde b_{n-k}  \left<\bar\partial_{k+\frac 1 2}^\alpha u, \bar\partial_{k+1} u\right> \geq 0 \quad \forall n\geq 0.
\end{eqnarray}
\end{lemma}

\begin{proof}
We provide a brief schematic of the proof.
Rewriting \eqref{eq:L1halfbound} into a quadratic form (omitting details here), it is sufficient to prove that
\begin{equation}
\mathbf A =
\left[\begin{array}{ccccc} \tilde b_0 &  &  &  &  \\ \tilde b_1 & \tilde b_0 &  &  &  \\\vdots & \vdots & \ddots &  &  \\ \tilde b_{n-1} & \tilde b_{n-2} & \cdots & \tilde b_0 &  \\ \tilde b_n & \tilde b_{n-1} & \cdots & \tilde b_1 & \tilde b_0
\end{array}\right]
\end{equation}
is positive definite.
Using the facts $2\tilde b_0>\tilde b_1$ and $2\tilde b_0 + \tilde b_2 >2\tilde b_1$, it is easy to verify that $\mathbf A + \mathbf A^{\rm T}$ satisfies the three conditions in Lemma \ref{lem:cholesky} and is therefore positive definite.
Thus, $\mathbf A$ is positive definite.

Similarly, to prove \eqref{eq:L1halfprop} it is sufficient to prove that
\begin{equation}
\mathbf B =
\left[\begin{array}{ccccc} \tilde b_n &  &  &  &  \\ & \tilde b_{n-1} &  &  &  \\ &  & \ddots &  &  \\ &  &  & \tilde b_1 &  \\ &  &  &  & \tilde b_0 \end{array}\right]
\left[\begin{array}{ccccc} \tilde b_0 &  &  &  &  \\ \tilde b_1 & \tilde b_0 &  &  &  \\\vdots & \vdots & \ddots &  &  \\ \tilde b_{n-1} & \tilde b_{n-2} & \cdots & \tilde b_0 &  \\ \tilde b_n & \tilde b_{n-1} & \cdots & \tilde b_1 & \tilde b_0\end{array}\right]
\end{equation}
is positive definite.
We consider the following conjugate transformation of $\mathbf B + \mathbf B^{\rm T}$:
\begin{equation}\label{eq:conj0}
 {\mathbf S} = P \left( {\mathbf B} +{\mathbf B}^{\rm T} \right) P^{\rm T},
\end{equation}
where $P$ is an anti-diagonal matrix
\begin{equation}\label{eq:Pn0}
P = \left[ \begin{array}{cccc} &  &  & \tilde b_{0}^{-1} \\ &  & \tilde b_{1}^{-1} &  \\ & \adots &  &  \\ \tilde b_{n}^{-1} &  &  & \end{array}\right].
\end{equation}
Similar to the proof of Lemma \ref{lem:L1prop1}, one can also verify that $\mathbf S$ satisfies the three conditions in Lemma \ref{lem:cholesky} and is therefore positive definite.
In particular, we used the facts that $2 \tilde b_0 \tilde b_2\geq \tilde b_1^2$ and $b_{j-1}b_{j+1}\geq b_j^2$ for $j\geq 2$, which are useful in the verification.
The details are omitted here.
\end{proof}

\vskip .25cm

By taking the inner product of \eqref{eq:L1SAV2} with $\bar\partial_{n+1} \phi$ and multiplying \eqref{eq:L1SAV3} with $2r^{n+\frac 1 2}$, we can obtain
\begin{equation}\label{eq:pn1E}
\bar\partial_{n+1} E \leq \left<\mu^{n+\frac 1 2}, \bar\partial_{n+1} \phi\right>,
\end{equation}
where $E^n = \frac 1 2 \left<\mathcal L \phi^n, \phi^n \right> + (r^n)^2$ is the modified energy.
Combining \eqref{eq:L1SAV1}, \eqref{eq:L1halfbound}, \eqref{eq:L1halfprop}, and \eqref{eq:pn1E}, we then derive the energy boundedness for the modified energy as stated below.

\vskip .25cm

\begin{theorem}[Energy boundedness \& Fractional energy law]\label{main-thm3}
For the ($2-\alpha$)-order L1-SAV scheme \eqref{eq:L1SAV}, the energy boundedness
\begin{equation}\label{eq:enebound1}
E^{n+1}\leq E^0 \qquad \forall n\geq 0,
\end{equation}
and the fractional energy law
\begin{equation}\label{eq:fraclaw1}
\bar\partial_{n+\frac 1 2} ^\alpha  E = \sum_{k=0}^{n} \tilde b_{n-k} \bar\partial_{k+1} E \leq 0 \qquad \forall n\geq 0
\end{equation}
hold true, where $E^n$ is the modified energy.
\end{theorem}

\vskip .25cm

We point out that
in the case of $\tilde b_0 \geq \tilde b_1$, i.e., $\alpha\geq \ln_3\frac 3 2 \approx 0.3691$,  the fractional energy law \eqref{eq:fraclaw1} can lead directly to the energy boundedness \eqref{eq:enebound1} by induction.

\section{Nonuniform time steps}\label{sect4}~
We will demonstrate in this section that Lemma \ref{lem:cholesky} can be extended to handle nonuniform time steps. To this end, consider the general nonuniform time mesh in the form
\begin{equation}\label{eq:tau_j}
\tau_j = t_j -t_{j-1},\quad 1\leq  j \leq N,
\end{equation}
where $\tau_j$ denotes the $j$th time step.
For the nonuniform time mesh \eqref{eq:tau_j}, the corresponding L1 approximation to $\partial_t^\alpha \phi$ at $t_n$ and $\frac 1 2 (t_{n+1}+t_n)$ becomes
\begin{equation}\label{eq:L1nonuni}
D_{n}^\alpha  \phi  = \sum_{j=1}^n  d_{n,j} D_{j} \phi, \qquad
D_{n+\frac 1 2}^\alpha  \phi = \sum_{j=0}^n \bar d_{n,j} D_{j+1}  \phi,
\end{equation}
respectively, where
\begin{subequations}
	\begin{align}
& d_{n,j}= \frac{(t_n-t_{j-1})^{1-\alpha} - (t_n-t_j)^{1-\alpha}}{ \Gamma(2-\alpha) \tau_j}, \\
& \bar d_{n,j}  = \frac{\left( t_{n+1}+t_n-2t_{j}\right)^{1-\alpha} - \left( t_{n+1}+t_n-2t_{j+1}\right)^{1-\alpha} }{ \Gamma(2-\alpha)2^{1-\alpha} \tau_{j+1}},\\
&\bar d_{n,n} = \frac{(t_{n+1}-t_n)^{1-\alpha}}{ \Gamma(2-\alpha) 2^{1-\alpha} \tau_{n+1}},
\end{align}
\end{subequations}
and $D_j \phi \coloneqq \phi^{j} - \phi^{j-1}$.
Note that the notations $D_{n}^\alpha  \phi , D_{n+\frac 1 2}^\alpha  \phi, D_j \phi, d_{n,j}, \bar d_{n,j}$ correspond respectively to the previous notations $\overline \partial_{n}^\alpha  \phi, \bar\partial_{n+\frac 1 2}^\alpha  \phi, \overline \partial_j \phi, b_{n-j}, \tilde b_{n-j}$, but now for nonuniform time meshes.

\vskip .25cm

\begin{lemma}\label{lem:L1nu}
For any function $u$ defined on $[0,T]\times  \Omega$, the following properties hold:
\begin{eqnarray}
&& \label{eq:L1nubound}
\sum_{k=1}^{n} \left<D_{k}^\alpha u, D_k u\right> \geq 0 \quad \forall n\geq 1,
\\
&& \label{eq:L1halfnubound}
\sum_{k=0}^{n} \left<D_{k+\frac 1 2}^\alpha u, D_{k+1} u \right> \geq 0 \quad \forall n\geq 0,
\end{eqnarray}
where $D_{k}^\alpha$ and $D_{k+\frac 1 2}^\alpha$ are given by \eqref{eq:L1nonuni} for nonuniform time steps.
\end{lemma}
\begin{proof}
To derive \eqref{eq:L1nubound}, we need to prove that
\begin{equation}
\mathbf D =
\left[\begin{array}{ccccc} d_{1,1} &  &  &  &  \\ d_{2,1} & d_{2,2} &  &  &  \\\vdots & \vdots & \ddots &  &  \\ d_{n-1,1} & d_{n-1,2} & \cdots & d_{n-1,n-1} &  \\ d_{n,1} & d_{n,2} & \cdots & d_{n,n-1} & d_{n,n} \end{array}\right]
\end{equation}
is positive definite.
Similarly, to derive \eqref{eq:L1halfnubound} we need to show that
\begin{equation}
\widetilde{\mathbf D} =
\left[\begin{array}{ccccc} \bar d_{0,0} &  &  &  &  \\ \bar d_{1,0} & \bar d_{1,1} &  &  &  \\\vdots & \vdots & \ddots &  &  \\ \bar d_{n-1,0} & \bar d_{n-1,1} & \cdots & \bar d_{n-1,n-1} &  \\ \bar d_{n,0} & \bar d_{n,1} & \cdots & \bar d_{n,n-1} & \bar d_{n,n} \end{array}\right]
\end{equation}
is positive definite.
Without much difficulty, one can verify that $\mathbf D + \mathbf D^{\rm T}$ and $\widetilde{\mathbf D}+ \widetilde{\mathbf D}^{\rm T}$ both satisfy the three conditions in Lemma \ref{lem:cholesky} and are therefore positive definite.
Here, we show details on the positive definiteness of $\mathbf S$:
\begin{eqnarray}
\mathbf S_{ij} &=& \Gamma(2-\alpha)\left(\mathbf D + \mathbf D^{\rm T}\right)_{ij} \nonumber \\
&=& \left\{
  \begin{array}{r@{}l}
	\begin{aligned}
	& 2 \tau_i^{-\alpha} && \mbox{if } i = j,\\
	& \tau_j^{-1} \left[(t_i-t_{j-1})^{1-\alpha} - (t_i-t_j)^{1-\alpha}\right]  && \mbox{if } i> j,\\
	& \mathbf S_{ji}  && \mbox{if } i< j.
	\end{aligned}
  \end{array}
  \right.
\end{eqnarray}
We focus on the lower triangular part of $\mathbf S$, i.e., $i\geq j$.
Firstly, it is easy to see that $\mathbf S_{ij}$ decreases w.r.t. $i$, i.e., the property (P1) in Lemma \ref{lem:cholesky}.
Secondly, we have
\begin{eqnarray}\label{eq:Sij_int}
\mathbf S_{ij}
& =& \tau_j^{-1} \left[(t_i-t_{j-1})^{1-\alpha} - (t_i-t_j)^{1-\alpha}\right] \nonumber \\
& =& \frac 1 {(1-\alpha)(t_j-t_{j-1})} \int_{t_{j-1}}^{t_j} (t_i-s)^{-\alpha}\,{\rm d}s \nonumber \\
& =& \frac{1}{1-\alpha}(t_i-\xi_j)^{-\alpha}, \qquad \xi_j \in (t_{j-1},t_j),
\end{eqnarray}
implying that $\mathbf S_{ij}$ increases w.r.t. $j$, i.e., the property (P2) in Lemma \ref{lem:cholesky}.
Thirdly, from \eqref{eq:Sij_int}, we have
\begin{eqnarray}
\mathbf S_{i-1,j} -\mathbf S_{i,j}
& =& \frac 1 {(1-\alpha)(t_j-t_{j-1})} \int_{t_{j-1}}^{t_j} \left[(t_{i-1}-s)^{-\alpha} - (t_i-s)^{-\alpha}\right]\,{\rm d}s \nonumber \\
& =& \frac{1}{1-\alpha}\left[(t_{i-1}-\eta_j)^{-\alpha} - (t_i-\eta_j)^{-\alpha}\right], \quad \eta_j \in (t_{j-1},t_j),
\end{eqnarray}
implying that $\mathbf S_{i-1,j}-\mathbf S_{i,j}$ increases w.r.t. $j$, i.e., the property (P3) in Lemma \ref{lem:cholesky} holds.
Therefore, $\mathbf S$ is positive definite and consequently, $\mathbf D$ is positive definite.

Similar proof can be done for $\widetilde{\mathbf D}$.
We shall verify that  $\widetilde{\mathbf S}= \Gamma(2-\alpha)(\widetilde{\mathbf D} + \widetilde{\mathbf D}^{\rm T})$ satisfies the three conditions in Lemma \ref{lem:cholesky}.
The proof is almost the same as before, except the verification of
\begin{equation}\label{ineq:tildeS}
\widetilde{\mathbf S}_{j,j} -\widetilde{\mathbf S}_{j+1,j} \geq \widetilde{\mathbf S}_{j,j-1} -\widetilde{\mathbf S}_{j+1,j-1} \qquad \forall j\geq 2.
\end{equation}
It can be verified that (\ref{ineq:tildeS}) is equivalent to
\begin{align}
	 &2\tau_{j+1}^{-\alpha} - \tau_{j+1}^{-1}\left[(\tau_{j+2}+2\tau_{j+1})^{1-\alpha}-\tau_{j+2}^{1-\alpha}\right] \label{ineq:ab} \\
&\geq \tau_{j}^{-1}\left[(\tau_{j+1}+2\tau_{j})^{1-\alpha}-\tau_{j+1}^{1-\alpha}\right]
- \tau_j^{-1} \left[(\tau_{j+2}+2\tau_{j+1}+2\tau_j)^{1-\alpha} - (\tau_{j+2}+2\tau_{j+1})^{1-\alpha}\right]. \nonumber
\end{align}
Let
\begin{eqnarray*}
Q(a,b) & =&  2 a^{1-\alpha}  -(b+2a)^{1-\alpha} + b^{1-\alpha} \label{de1} \\
&& \quad  - a\left[(a+2)^{1-\alpha}-a^{1-\alpha}-(b+2a+2)^{1-\alpha}+(b+2a)^{1-\alpha}\right]. \nonumber
\end{eqnarray*}
For any $a,b>0$, straightforward computation gives
\[
\partial_b Q(a,b) = (1-\alpha) (a+1) \left[\frac{1}{a+1}b^{-\alpha} + \frac{a}{a+1}(b+2a+2)^{-\alpha}-(b+2a)^{-\alpha}\right]\geq 0,
\]
where the Jensen's inequality is used. Further, when $b=0$, we have $Q(a,0) = a^{2-\alpha} p(1/a)$, where
\[
p(x) = (1+ 2x) - (1+2x)^{1-\alpha} +2^{1-\alpha} (1+x)^{1-\alpha}-2^{1-\alpha}(1+x).
\]
It is easy to find that $p(0) = 0$, $p'(0) = 0$ and $p''(x) \geq 0$, so that $p(x)\geq 0$.
Since $Q(a,0)\geq 0$ and $\partial_b Q(a,b)\geq 0$, we have $Q(a,b)\geq 0$ for any $a,b>0$. In particular, we have
\[
 \tau_{j+1} ^{-1} \tau_j^{1- \alpha} Q(  \tau_{j+1}\tau_{j}^{-1}, \tau_{j+2}\tau_{j}^{-1}) \ge 0.
 \]
It can be verified that the above inequality is equivalent to \eqref{ineq:ab}. This completes the proof of the lemma.
\end{proof}

\vskip .3cm

With nonuniform time mesh, we rewrite the first-order stabilized L1 scheme \eqref{eq:L1STA} as
\begin{equation}\label{sch:L1nu1}
D_{n+1}^\alpha \phi = \gamma \,\mathcal G \left(\mathcal L \phi^{n+1} +  \delta_\phi E_{1}(\phi^n) + \widetilde{\mathcal L} \left(\phi^{n+1}-\phi^n\right) \right),
\end{equation}
and the ($2-\alpha$)-order L1-SAV scheme \eqref{eq:L1SAV} as
\begin{subequations}\label{sch:L1nu2}
\begin{align}
& D_{n+\frac 1 2} ^\alpha \phi = \gamma \mathcal G \mu^{n+\frac 1 2}, \label{sch:L1nu21} \\
& \mu^{n+\frac 1 2} = \mathcal L \phi^{n+\frac 1 2} + \frac{r^{n+\frac 1 2}}{\sqrt{E_1(\bar \phi^{n+\frac 1 2})+C_0}} \delta_\phi E_1(\bar \phi^{n+\frac 1 2}), \label{sch:L1nu22} \\
& r^{n+1}-r^{n}  = \frac{1}{2\sqrt{E_1(\bar \phi^{n+\frac 1 2})+C_0}} \left<\delta_\phi E_1(\bar \phi^{n+\frac 1 2}), \phi^{n+1}-\phi^n\right>. \label{sch:L1nu23}
\end{align}
\end{subequations}
where $\mathcal L$, $\widetilde{\mathcal L}$, $E_1$, $\overline\phi^{n+\frac 1 2}$, and $C_0$ are the same as before.

Using Lemma \ref{lem:L1nu} we can derive the energy bound preserving property of the L1 schemes  \eqref{sch:L1nu1} and \eqref{sch:L1nu2} with nonuniform time meshes. The details will be omitted here.

\vskip .25cm

\begin{theorem}[Energy boundedness]\label{thm-nu}
The energy boundedness holds for the L1 schemes \eqref{sch:L1nu1} and \eqref{sch:L1nu2} with arbitrary nonuniform time mesh:
\begin{equation}
E^n \leq E^0 \quad \forall n\geq 1,
\end{equation}
where $E^n$ denotes the classical discrete energy for \eqref{sch:L1nu1} and the modified SAV energy for \eqref{sch:L1nu2}.
\end{theorem}
}

\vskip .25cm
%\begin{remark}
%The authors in \cite{} propose an L1+ approximation in order to preserve the energy boundedness property for nonuniform time meshes.
%We have shown for the first time (to our knowledge) that the classical L1 schemes with nonuniform meshes for the time-fractional phase-field equation preserve energy boundedness property.
%\end{remark}

%\vskip .25cm

%Note that with the classical time derivative, high order energy stable schemes for phase-field equations have arisen much interest, for example, the second-order convex-splitting schemes for the CH equation \cite{guo2016h,diegel2016stability} and the MBE equation \cite{shen2012second} with uniform time meshes, the BDF2 convex-splitting scheme for CH equation \cite{chen2019second} with nonuniform time meshes.
%We mention that some second order schemes using the L2 approximation \cite{lv2016error} could still preserve the energy boundedness, which will be presented in our future work.

\section{Numerical tests}
\label{sect5}
In this section, we test the proposed schemes for time-fractional phase-field equations to verify our energy stability results.
In particular, we test the first order stabilized L1 scheme \eqref{eq:L1STA} for the time-fractional AC and CH models, and then the ($2-\alpha$)-order L1-SAV scheme \eqref{eq:L1SAV} for the time-fractional MBE model with slope selection.
The implementations of these schemes are linearly implicit without using nolinear iterations.

\subsection{Time-fractional  Allen--Cahn model}
We first test the stabilized L1 scheme \eqref{eq:L1STA} for the time fractional AC model defined in a two-dimensional domain $\Omega = [0,L_{x}] \times [0,L_{y}]$ with periodic boundary conditions.
%We mention that the fast sum-of-exponential algorithm \cite{jiang2017fast} could be used for evaluating the time-fractional derivatives.
We take $L_x = L_y = 2,~ \varepsilon = 0.1$, and $\gamma = 1$ in \eqref{eq:phase_field}.
The stabilization constant $S$ in scheme \eqref{eq:L1STA} is set to be $S = 2$ and the time step is set to be $\Delta t = 0.01$.
The peuso-spectral method with $128 \times 128$ Fourier modes is used for space discretization.
The initial phase-field state is taken as
\begin{equation}
  \phi_0(x) = \tanh\left[ \frac{1}{2\varepsilon} \left(\frac {2r} 3 - \frac 1 4 - \frac{1+\cos(6\theta)}{16} \right) \right]
\end{equation}
with the polar coordinates
$r = \sqrt{x^2+y^2}$ and $\theta = \arctan ( y/ x)$.

Fig. \ref{fig:snapshots_AC} illustrates the solution $\phi$ to the time-fractional AC equation with different orders of derivative
$\alpha = 1, 0.8, 0.5, 0.3$.
%It is observed that when $\alpha$ becomes smaller, it takes more time to reach the equilibrium.
On the left-hand side of Fig. \ref{fig:energy_AC}, one can see that the energy decreases with respect to time.
In the middle of Fig. \ref{fig:energy_AC}, it can be observed that the fractional derivative of energy is always nonpositive for different values of $\alpha$, which is in good agreement with the discrete fractional energy law in Theorem \ref{main-thm1}.
On the right-hand side of Fig. \ref{fig:energy_AC}, we plot the derivative of wighted energy with respect to time, which is found to be nonpositive as stated in Theorem \ref{main-thm2}.

\begin{figure}[!h]
\centering
\includegraphics[trim={0.5in 1in 2in 0.5in},clip,width=\textwidth]{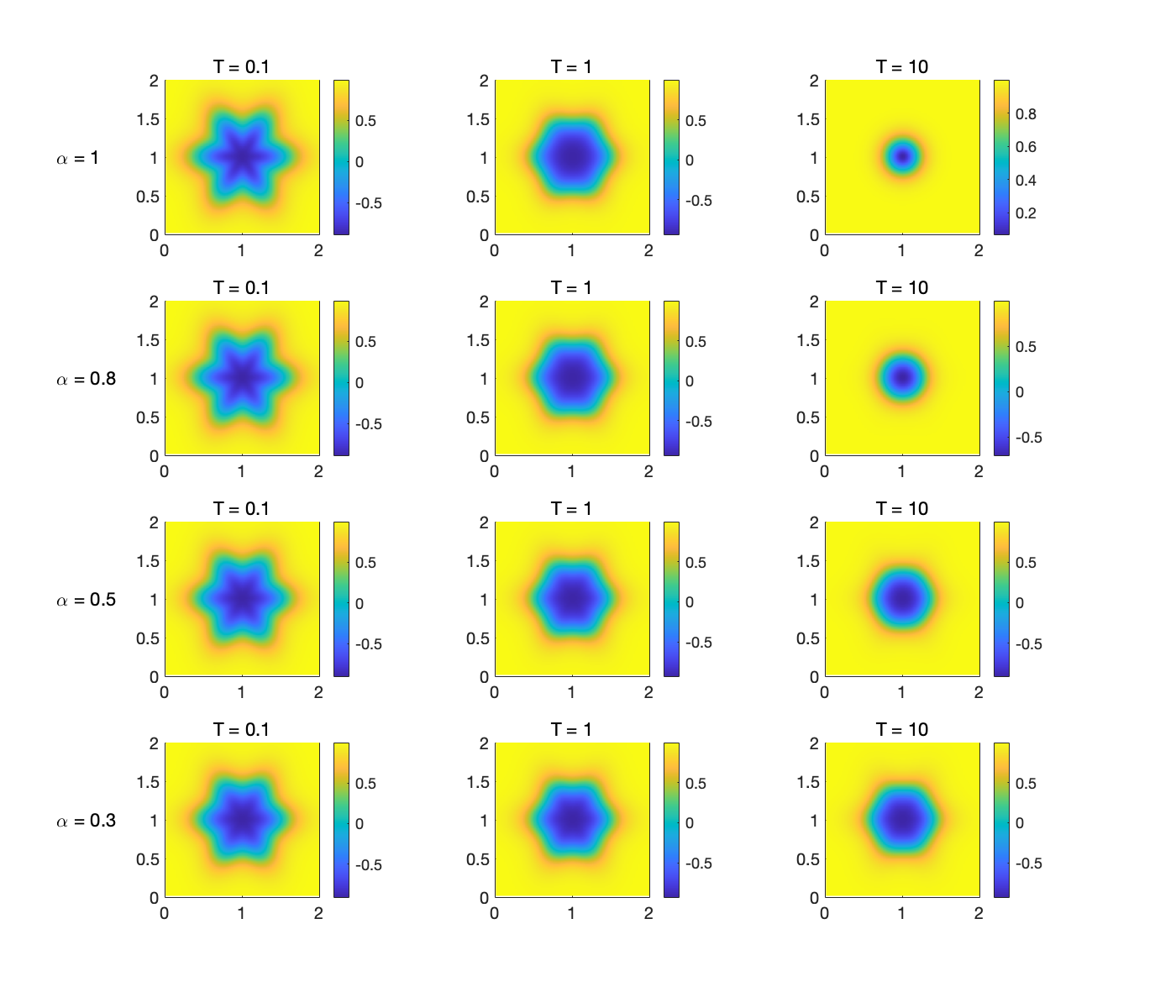}
\vspace{-0.3in}
\caption{Snapshots of the time-fractional Allen--Cahn solution with  $\alpha = 1, 0.8, 0.5, 0.3$.} \label{fig:snapshots_AC}
\end{figure}

\begin{figure}[!h]
\centering
\includegraphics[trim={1.6in 0.in 1.6in 0.in},clip,width=\textwidth]{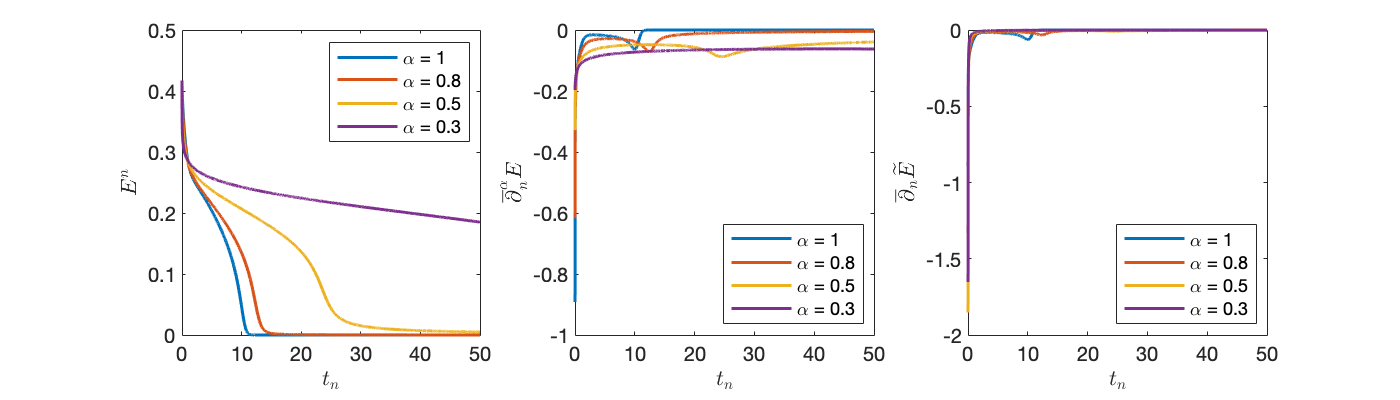}
\vspace{-0.3in}
\caption{Energy (left), time-fractional derivative of energy (middle), and time derivative of weighted energy (right) w.r.t. time,  for the time-fractional Allen--Cahn model with $\alpha = 1, 0.8, 0.5, 0.3$.} \label{fig:energy_AC}
\end{figure}

\subsection{Time-fractional Cahn--Hilliard model}
For the time-fractional CH model defined on $ [0,L_{x}] \times [0,L_{y}]$, we again solve the governing equation using the stabilization scheme \eqref{eq:L1STA}. The following parameters are used: $L_x = L_y = 2, \varepsilon = 0.1$, $\gamma = 0.1$, $L = 8$, $S = 4$, and $\Delta t = 0.01$. Still, $128 \times 128$ Fourier modes are used in the peuso-spectral method.
The initial phase-field state $\phi_0$ is taken as the uniformly distributed random field in $[-1,1]$.

Fig. \ref{fig:snapshots_CH} illustrates the phase-field function $\phi$ with $\alpha = 1, 0.8, 0.5, 0.3$.
We observed an interesting phenomena that the steady state with with $\alpha=0.5$ and $0.3$is very different with that of $\alpha=1$ and $0.8$.
In addition, it is observed from the left-hand side of Fig. \ref{fig:energy_CH} the energy decreases with respect to time although its theoretical justification is still unavailable.
Furthermore, in the middle of Fig. \ref{fig:energy_CH}, it can be observed that the fractional derivative of energy is always nonpositive as stated in Theorem \ref{main-thm1}.
%That is to say, the fractional energy law is satisfied in this numerical test.
On the right-hand side of Fig. \ref{fig:energy_CH}, the derivative of weighted energy is nonpositive as stated in Theorem \ref{main-thm2}.

\begin{figure}[!h]
\centering
\includegraphics[trim={0.5in 1in 2in 0.5in},clip,width=\textwidth]{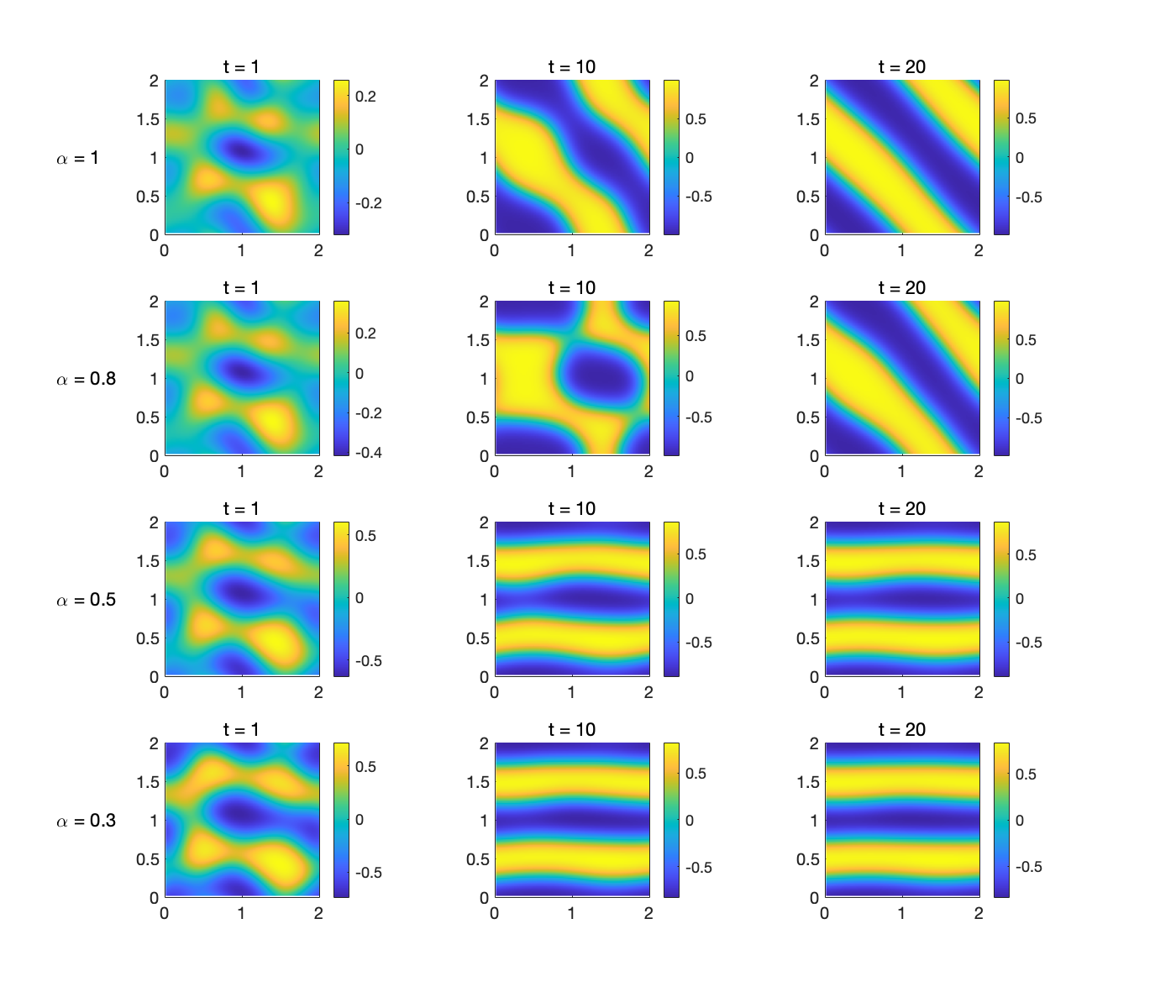}
\vspace{-0.3in}
\caption{Snapshots of the time-fractional Cahn--Hilliard solution with  $\alpha = 1, 0.8, 0.5, 0.3$.} \label{fig:snapshots_CH}
\end{figure}

\begin{figure}[!h]
\centering
\includegraphics[trim={1.6in 0.in 1.6in 0.in},clip,width=\textwidth]{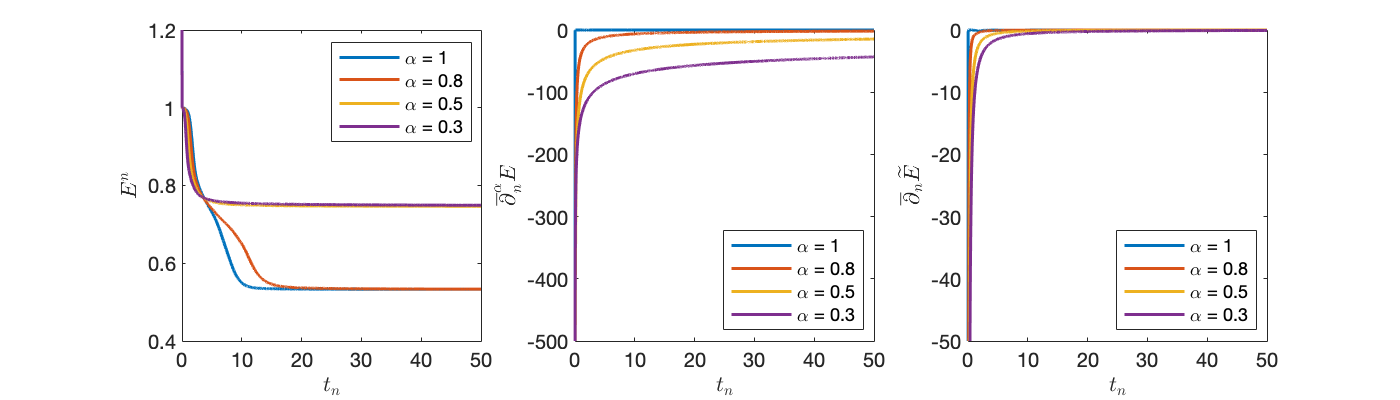}
\vspace{-0.3in}
\caption{Energy (left), time-fractional derivative of energy (middle), and time derivative of weighted energy (right) w.r.t. time $t$,  for the time-fractional Cahn--Hilliard equation with $\alpha = 1, 0.8, 0.5, 0.3$.} \label{fig:energy_CH}
\end{figure}

\subsection{Time-fractional MBE model with slope selection}
We test the ($2-\alpha$)-order L1-SAV schemes \eqref{eq:L1SAV} and \eqref{sch:L1nu2} for solving the time-fractional MBE equation \eqref{eq:mbe} with slope selection  defined in the domain $ [0,2\pi] \times [0,2\pi]$.
The following parameters are used: $\varepsilon^2 = 0.1$, $\gamma = 1$, and $C_0 = 1$.
For the peuso-spectral method, $128 \times 128$ Fourier modes are used.
The initial phase-field state $\phi_0$ is taken to be
\begin{equation}
\phi_0(x,y) = 0.1 \left[\sin(3x)\sin(2y)+\sin(5x)\sin(5y)\right],\quad (x,y)\in [0,2\pi]^2.
\end{equation}
First, we test the scheme \eqref{eq:L1SAV}  with uniform time step $\Delta t = 0.01$.
Fig. \ref{fig:snapshots_MBE} illustrates the phase solution $\phi$ to the time-fractional MBE equation with $\alpha = 1, 0.8, 0.5$ and $0.3$.
It can be observed that when $\alpha$ becomes smaller, the phase changes faster at the beginning but later changes quite slowly.
On the left-hand side of Fig. \ref{fig:energy_MBE}, the modified energy decreases w.r.t. time in the case of $\alpha = 1, 0.8$ and $0.5$.
However, the energy dissipation is violated at $t_n \approx 1.44$ in the case of $\alpha = 0.3$, as observed in \cite{hou2019variant}.
The violation occurs for small $\alpha$, which requires much smaller mesh sized to obtain satisfactory resolution.
Even in this case it is observed that the energy stability (boundedness) is preserved.
In the middle of Fig. \ref{fig:energy_MBE}, the fractional energy law in Theorem \ref{main-thm3} is verified.
On the right-hand side of the figure, it is found that the derivative of weighted (modified) energy is nonpositive.

\begin{figure}[!h]
\centering
\includegraphics[trim={0.5in 1in 2in 0.5in},clip,width=\textwidth]{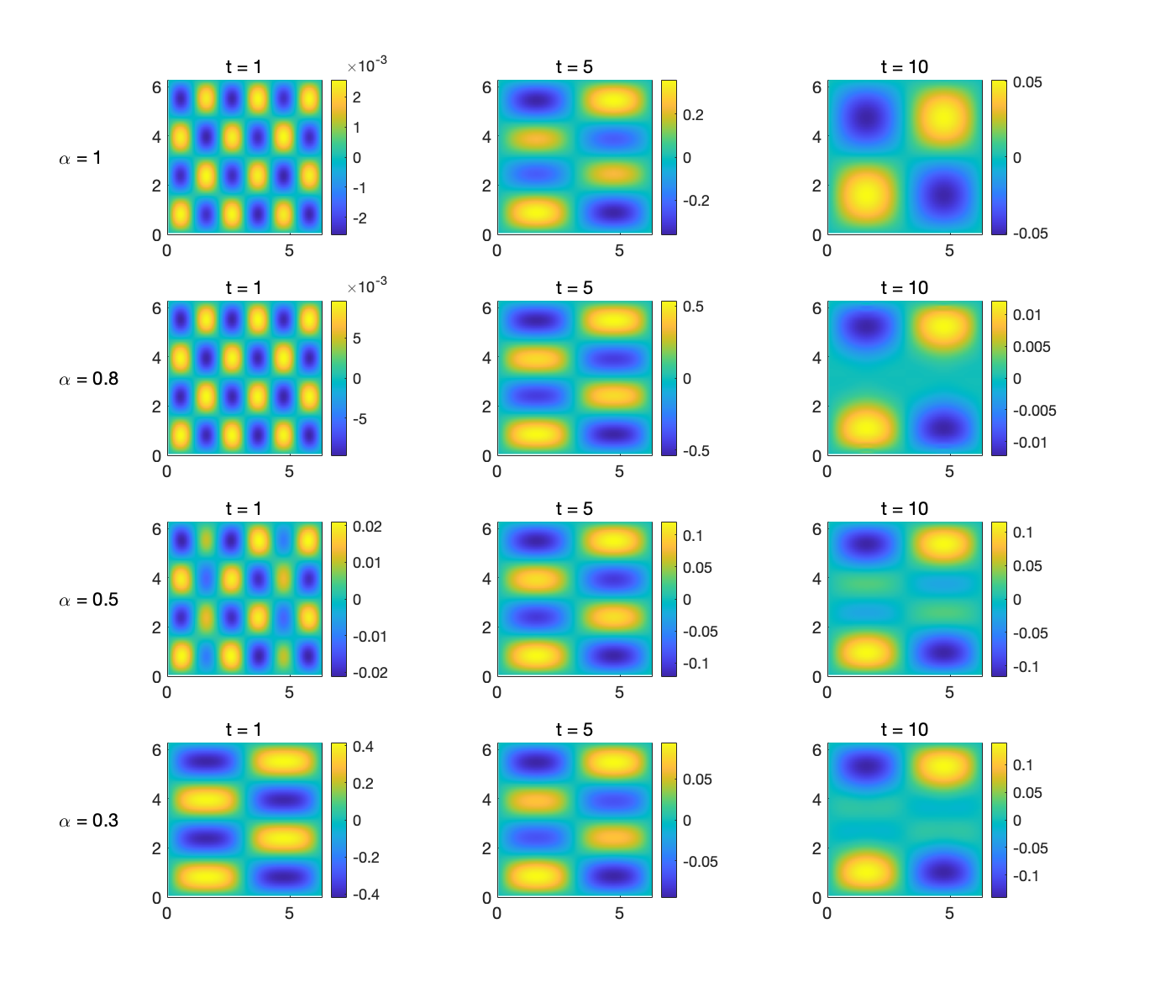}
\vspace{-0.3in}
\caption{Snapshots of the solution to the time-fractional MBE equation with slope selection with $\alpha = 1, 0.8, 0.5, 0.3$.} \label{fig:snapshots_MBE}
\end{figure}

\begin{figure}[!h]
\centering
\includegraphics[trim={.5in 0.in 1.6in 0.in},clip,width=\textwidth]{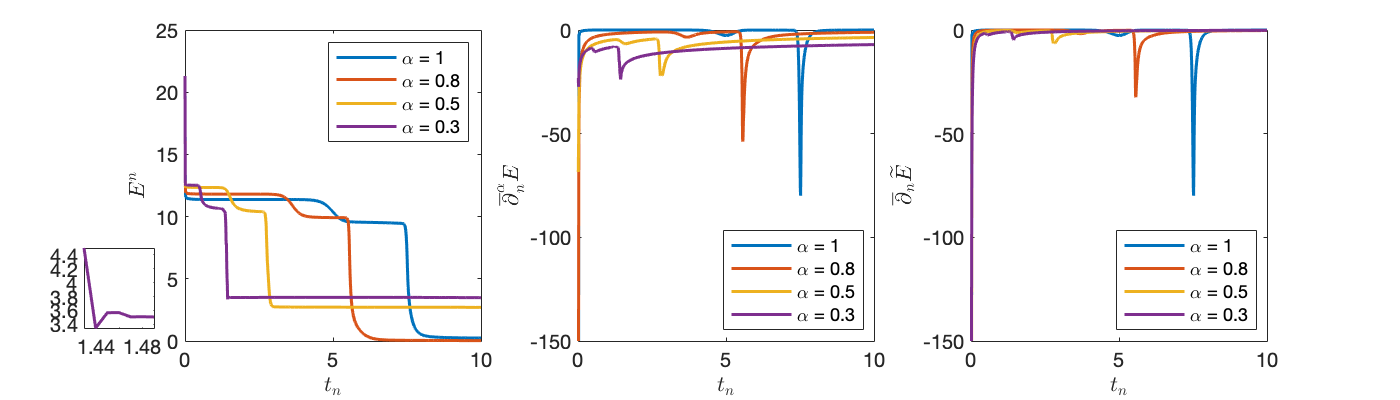}
\vspace{-0.3in}
\caption{Energy (left), time-fractional derivative of energy (middle), and time derivative of weighted energy (right) w.r.t. time $t$, for the time-fractional MBE equation with slope selection with $\alpha = 1, 0.8, 0.5, 0.3$. Here, in the middle and the right-hand side figures, the $y$-axis is cut for a better illustration.} \label{fig:energy_MBE}
\end{figure}

Next, we test the scheme \eqref{sch:L1nu2} with graded time mesh (see, e.g., \cite{tanggraded92}):
\begin{equation}
t_j = \left({j}/{N}\right)^{r} T, \quad j = 0,1,\ldots,N,
\end{equation}
where $r\geq 1$ is some constant and $N$ is the total number of time steps.
In our test, we set $r=1.2$ and $T=10$.
Fig. \ref{fig:energy_MBE_graded} illustrate the modified energy for $N = 100,~500$ and $1000$ respectively, which corresponds
roughly to $\Delta t \approx 0.1, 0.02, 0.01$ away from $t=0$.
It is observed that the computed energy is bounded above by the initial energy, as stated in Theorem \ref{thm-nu}.
One can find that in the case of $N = 100$ and $\alpha = 0.3$, the computed energy oscillates despite that the stability can be ensured.
We also mention that when $\alpha = 0.3$ and the computed energy at $T = 10$ varies  with different values of $N$.

\begin{figure}[!h]
\centering
\includegraphics[trim={1.6in 0.in 1.6in 0.in},clip,width=\textwidth]{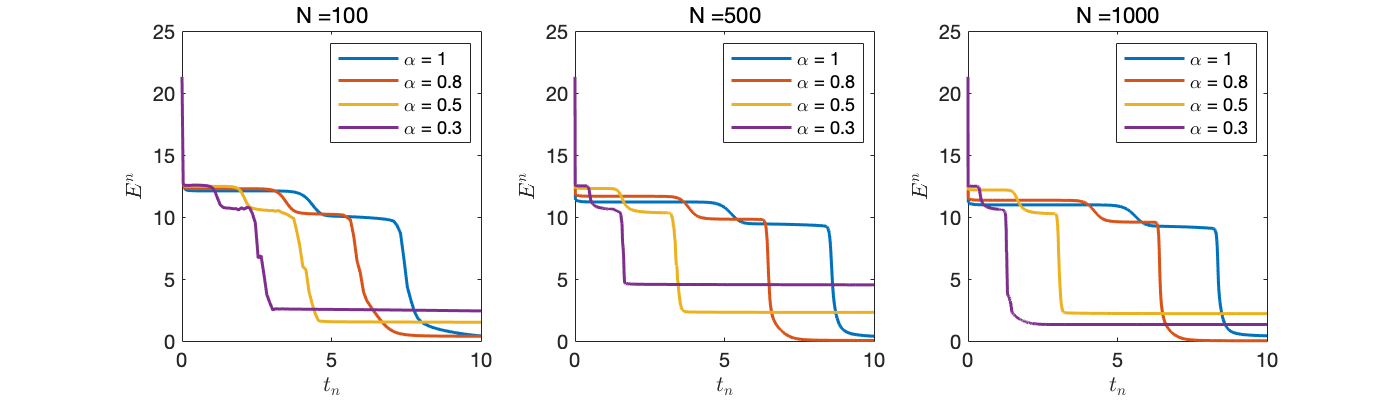}
\vspace{-0.3in}
\caption{Discrete energy of the L1-SAV scheme with graded time mesh for the time-fractional MBE equation with slope selection, where $N = 100, 500, 1000$ and  $\alpha = 1, 0.8, 0.5, 0.3$.} \label{fig:energy_MBE_graded}
\end{figure}

\section{Conclusion}
In this work, we applied a special Cholesky decomposition technique to analyze the energy stability of the first-order L1 approximation for time-fractional phase-field equations. From a numerical point view, the essential step to study the fractional PDE is to approximate the time-fractional derivative operator. We use a classical scheme called L1 approximation, which is naturally derived from the approximation of the fractional integral appropriately. In particular, we investigated the discrete versions of the fractional energy law of Du et al. \cite{du2020time}
and the weighted energy lay of \cite{quan2020define}, and obtained the dissipation of the fractional/weighted energy associated with numerical schemes. A higher order numerical scheme, i.e., a ($2-\alpha$)-order L1-SAV scheme, is also investigated, while
the relevant energy boundedness and the fractional energy law are established.
Moreover, we prove that the L1 schemes with nonuniform time steps still preserve the energy boundness.

We point out that the Cholesky decomposition technique is useful for the numerical schemes for time-fractional problems. It is expected that the framework developed in this work can also be employed to study second-order schemes for time-fractional phase-field equations. In fact, there have been extensive works of second-order
accurate energy stable numerical schemes for the Cahn-Hilliard
and MBE equation, using the convex splitting approach, in the standard temporal derivative case.
Both the Crank-Nicolson
and the modified BDF2 approaches, with the standard finite difference, mixed finite element and Fourier
pseudo-spectral spatial approximations have been reported,  see, e.g., \cite{diegel2016stability,lv2016error,shen2012second}.
It is certainly of some theoretical and numerical interests to extend these results to the time-fractional equations.
The corresponding analysis for the stability issues seems a challenging issue.

\section*{Acknowledgements}
This work is partially supported by the Special Project on High-Performance Computing of the National Key R\&D Program under No. 2016YFB0200604, NSFC Grant No. 11731006,  NSFC/Hong Kong RRC Grant 11961160718, and the fund of the Guangdong Provincial Key Laboratory of Computational Science and Material Design (No. 2019B030301001). The work of J. Yang is also supported by NSFC-11871264 and the Guangdong Basic and Applied Basic Research Foundation (2018A0303130123). The work of C. Quan is supported by NSFC-11901281, and the Guangdong Basic and Applied Basic Research Foundation (2020A1515010336).

\bibliography{bibfile_tang.bib}
\bibliographystyle{unsrt}

\end{document}